\RequirePackage{fix-cm}

\documentclass[smallextended]{svjour3}       
\smartqed  
\usepackage{graphicx,epsfig}

\usepackage{amsfonts,amsmath}
\usepackage{xspace}
\usepackage{color}
\usepackage{algorithm}
\usepackage{algpseudocode}
\usepackage{scalefnt}
\usepackage{tablefootnote}
\usepackage{enumerate}

\usepackage{nicefrac}

\usepackage{subfigure}

\usepackage{caption}

\usepackage{float}

\usepackage{url}

\begin{document}

\title{DC Decomposition of Nonconvex Polynomials\\ with Algebraic Techniques
}


\author{Amir Ali Ahmadi         \and
        Georgina Hall \thanks{The authors are partially supported by the Young Investigator Program Award of the AFOSR and the CAREER Award of the NSF.}
}


\institute{Amir Ali Ahmadi \at
             ORFE, Princeton University, Sherrerd Hall, Princeton, NJ 08540 \\
              \email{a\_a\_a@princeton.edu}           
          \and
         Georgina Hall \at
             ORFE, Princeton University, Sherrerd Hall, Princeton, NJ 08540\\
\email{gh4@princeton.edu}
}

\date{Received: date / Accepted: date}

\maketitle

\begin{abstract}
We consider the problem of decomposing a multivariate polynomial as the difference of two convex polynomials. We introduce algebraic techniques which reduce this task to linear, second order cone, and semidefinite programming. This allows us to optimize over subsets of valid difference of convex decompositions (dcds) and find ones that speed up the convex-concave procedure (CCP). We prove, however, that optimizing over the entire set of dcds is NP-hard.

%
\end{abstract}

\newenvironment{gh}{\color{black}}

\section{Introduction}\label{sec:intro}


A difference of convex (dc) program is an optimization problem of the form

\begin{equation}\label{eq:dcP}
\begin{aligned}
&\min f_0(x)\\
&\text{s.t. } f_i(x)\leq 0, i=1,\ldots,m,
\end{aligned}
\end{equation}
where $f_0,\ldots,f_m$ are difference of convex functions; i.e., 
\begin{equation} \label{eq:dceq}
\begin{aligned} 
&f_i(x)=g_i(x)-h_i(x),i=0,\ldots,m, \\
\end{aligned}
\end{equation} 
and $g_i:\mathbb{R}^n \rightarrow \mathbb{R},~h_i:\mathbb{R}^n \rightarrow \mathbb{R}$ are convex functions. The class of functions that can be written as a difference of convex functions is very broad containing for instance all functions that are twice continuously differentiable~\cite{Hartman}, \cite{hiriart}. Furthermore, any continuous function over a compact set is the uniform limit of a sequence of dc functions; see, e.g., reference \cite{Horst} where several properties of dc functions are discussed.


Optimization problems that appear in dc form arise in a wide range of applications. Representative examples from the literature include machine learning and statistics (e.g., kernel selection \cite{Argyriou}, feature selection in support vector machines \cite{le2008dc}, sparse principal component analysis \cite{lipp2014variations}, and reinforcement learning \cite{piot2014difference}), operations research (e.g., packing problems and production-transportation problems \cite{Tuy}), communications and networks \cite{pang},\cite{lou2015}, circuit design \cite{lipp2014variations}, finance and game theory \cite{gulpinar2010}, and computational chemistry \cite{floudas2013}. We also observe that dc programs can encode constraints of the type $x \in \{0,1\}$ by replacing them with the dc constraints $0 \leq x \leq 1, x-x^2 \leq 0$. This entails that any binary optimization problem can in theory be written as a dc program, but it also implies that dc problems are hard to solve in general.

As described in \cite{tao1997}, there are essentially two schools of thought when it comes to solving dc programs. The first approach is global and generally consists of rewriting the original problem as a concave minimization problem (i.e., minimizing a concave function over a convex set; see \cite{tuy1988}, \cite{tuy1987}) or as a reverse convex problem (i.e., a convex problem with a linear objective and one constraint of the type $h(x)\geq 0$ where $h$ is convex). We refer the reader to \cite{Tuy86} for an explanation on how one can convert a dc program to a reverse convex problem, and to \cite{HJ80b} for more general results on reverse convex programming. These problems are then solved using branch-and-bound or cutting plane techniques (see, e.g., \cite{Tuy} or \cite{Horst}). The goal of these approaches is to return global solutions but their main drawback is scalibility. The second approach by contrast aims for local solutions while still exploiting the dc structure of the problem by applying the tools of convex analysis to the two convex components of a dc decomposition. One such algorithm is the Difference of Convex Algorithm (DCA) introduced by Pham Dinh Tao in \cite{PDT76} and expanded on by Le Thi Hoai An and Pham Dinh Tao. This algorithm exploits the duality theory of dc programming~\cite{toland79} and is popular because of its ease of implementation, scalability, and ability to handle nonsmooth problems.



In the case where the functions $g_i$ and $h_i$ in (\ref{eq:dceq}) are differentiable, DCA reduces to another popular algorithm called the Convex-Concave Procedure (CCP) \cite{Lanckriet}. The idea of this technique is to simply replace the concave part of $f_i$ (i.e., $-h_i$) by a linear overestimator as described in Algorithm \ref{alg:CCP}. By doing this, problem (\ref{eq:dcP}) becomes a convex optimization problem that can be solved using tools from convex analysis. The simplicity of CCP has made it an attractive algorithm in various areas of application. These include statistical physics (for minimizing Bethe and Kikuchi free energy functions \cite{YR}), machine learning \cite{lipp2014variations},\cite{Fung},\cite{Chapelle}, and image processing \cite{Wang}, just to name a few. In addition, CCP enjoys two valuable features: (i) if one starts with a feasible solution, the solution produced after each iteration remains feasible, and (ii) the objective value improves in every iteration, i.e., the method is a descent algorithm.
%
The proof of both claims readily comes out of the description of the algorithm and can be found, e.g., in \cite[Section 1.3.]{lipp2014variations}, where several other properties of the method are also laid out. Like many iterative algorithms, CCP relies on a stopping criterion to end. This criterion can be chosen amongst a few alternatives. For example, one could stop if the value of the objective does not improve enough, or if the iterates are too close to one another, or if the norm of the gradient of $f_0$ gets small. 

\begin{algorithm}[H]
\caption{ CCP}
\label{alg:CCP}
\begin{algorithmic}[1]
\Require $x_0,~ f_i=g_i-h_i, i=0,\ldots,m$
\State $k\leftarrow 0$
\While{stopping criterion not satisfied}
\State Convexify: $f_i^{k}(x)\mathrel{\mathop{:}}=g_i(x)-(h_i(x_k)+\nabla h_i(x_k)^T(x-x_k)),~ i=0,\ldots,m$
\State Solve convex subroutine: $\min f_0^k(x)$, s.t. $f_i^k(x) \leq 0, i=1,\ldots,m$
\State $x_{k+1}\mathrel{\mathop{:}}= \underset{f_i^{k}(x) \leq 0}{\text{argmin}} f_0^k(x)$
\State $k \leftarrow k+1$
\EndWhile
\Ensure $x_k$
\end{algorithmic}
\end{algorithm}

Convergence results for CCP can be derived from existing results found for DCA, since CCP is a subcase of DCA as mentioned earlier. But CCP can also be seen as a special case of the family of majorization-minimization (MM) algorithms. Indeed, the general concept of MM algorithms is to iteratively upperbound the objective by a convex function and then minimize this function, which is precisely what is done in CCP. This fact is exploited by Lanckriet and Sriperumbudur in \cite{Lanckriet} and Salakhutdinov et al. in \cite{Salak} to obtain convergence results for the algorithm, showing, e.g., that under mild assumptions, CCP converges to a stationary point of the optimization problem (\ref{eq:dcP}).

\subsection{Motivation and organization of the paper}

Although a wide range of problems already appear in dc form (\ref{eq:dceq}), such a decomposition is not always available. In this situation, algorithms of dc programming, such as CCP, generally fail to be applicable. Hence, the question arises as to whether one can (efficiently) compute a difference of convex decomposition (dcd) of a given function. This challenge has been raised several times in the literature. For instance, Hiriart-Urruty~\cite{hiriart} states ``All the proofs [of existence of dc decompositions] we know are ``constructive'' in the sense that they indeed yield [$g_i$] and [$h_i$] satisfying (\ref{eq:dceq}) but could hardly be carried over [to] computational aspects''. As another example, Tuy~\cite{Tuy} writes: ``The dc structure of a given problem is not always apparent or easy to disclose, and even when it is known explicitly, there remains for the problem solver the hard task of bringing this structure to a form amenable to computational analysis.''


Ideally, we would like to have not just the ability to find one dc decomposition, but also to optimize over the set of valid dc decompositions. Indeed, dc decompositions are not unique: Given a decomposition $f=g-h$, one can produce infinitely many others by writing $f=g+p-(h+p),$ for any convex function $p$. This naturally raises the question whether some dc decompositions are better than others, for example for the purposes of CCP. 


In this paper we consider these decomposition questions for multivariate polynomials. Since polynomial functions are finitely parameterized by their coefficients, they provide a convenient setting for a computational study of the dc decomposition questions. Moreover, in most practical applications, the class of polynomial functions is large enough for modeling purposes as polynomials can approximate any continuous function on compact sets with arbitrary accuracy. It could also be interesting for future research to explore the potential of dc programming techniques for solving the polynomial optimization problem. This is the problem of minimizing a multivariate polynomial subject to polynomial inequalities and is currently an active area of research with applications throughout engineering and applied mathematics. In the case of quadratic polynomial optimization problems, the dc decomposition approach has already been studied~\cite{Bomze},\cite{quadratic_dc}.

With these motivations in mind, we organize the paper as follows. In Section~\ref{sec:undominated}, we start by showing that unlike the quadratic case, the problem of testing if two given polynomials $g,h$ form a valid dc decomposition of a third polynomial $f$ is NP-hard (Proposition~\ref{prop:dcd.NPhard}). We then investigate a few candidate optimization problems for finding dc decompositions that speed up the convex-concave procedure. In particular, we extend the notion of an undominated dc decomposition from the quadratic case~\cite{Bomze} to higher order polynomials. We show that an undominated dcd always exists (Theorem~\ref{thm:undom.dcd}) and can be found by minimizing a certain linear function of one of the two convex functions in the decomposition. However, this optimization problem is proved to be NP-hard for polynomials of degree four or larger (Proposition~\ref{prop:undominated.NPhard}). To cope with intractability of finding optimal dc decompositions, we propose in Section~\ref{sec:ConvRelax} a class of algebraic relaxations that allow us to optimize over subsets of dcds. These relaxations will be based on the notions of \emph{dsos-convex, sdsos-convex,} and \emph{sos-convex} polynomials (see Definition~\ref{def:alternatives.convexity}), which respectively lend themselves to {\gh\emph{linear, second order cone,}} and \emph{semidefinite programming}. In particular, we show that a dc decomposition can always be found by linear programming (Theorem~\ref{thm:diff.dsos}). Finally, in Section~\ref{sec:numerical.results}, we perform some numerical experiments to compare the scalability and performance of our different algebraic relaxations.

\section{Polynomial dc decompositions and their complexity} \label{sec:undominated}

To study questions around dc decompositions of polynomials more formally, let us start by introducing some notation. A multivariate \emph{polynomial} $p(x)$ in variables
$x\mathrel{\mathop:}=(x_1,\ldots,x_n)^T$ is a function from
$\mathbb{R}^n$ to $\mathbb{R}$ that is a finite linear combination
of monomials:
\begin{equation}
p(x)=\sum_{\alpha}c_\alpha x^\alpha=\sum_{\alpha_1, \ldots,
\alpha_n} c_{\alpha_1,\ldots,\alpha_n} x_1^{\alpha_1} \cdots
x_n^{\alpha_n} ,
\end{equation}
where the sum is over $n$-tuples of nonnegative
integers $\alpha_i$. 
The \emph{degree} of a monomial $x^\alpha$ is equal to $\alpha_1 +
\cdots + \alpha_n$. The degree of a polynomial $p(x)$ is defined to
be the highest degree of its component monomials. A simple counting
argument shows that a polynomial of degree $d$ in $n$ variables has
$\binom{n+d}{d}$ coefficients. A \emph{homogeneous polynomial} (or a
\emph{form}) is a polynomial where all the monomials have the same
degree. An $n$-variate form $p$ of degree $d$ has $\binom{n+d-1}{d}$ coefficients. {\gh We denote the set of polynomials (resp. forms) of degree $2d$ in $n$ variables by $\tilde{\mathcal{H}}_{n,2d}$ (resp. $\mathcal{H}_{n,2d}$). }

Recall that a symmetric matrix $A$ is positive semidefinite (psd) if $x^TAx \geq 0$ for all $x \in \mathbb{R}^n$; this will be denoted by the standard notation $A \succeq 0.$ Similarly, a polynomial $p(x)$ is said to be \emph{nonnegative} or
positive semidefinite if $p(x)\geq0$ for all $x\in\mathbb{R}^n$.
For a polynomial $p$, we denote its Hessian by $H_p$. The second order characterization of convexity states that $p$ is convex if and only if $H_p(x) \succeq 0$, $\forall x \in\mathbb{R}^n.$

\begin{definition}
We say a polynomial $g$ is a \emph{dcd} of a polynomial $f$ if $g$ is convex and $g-f$ is convex. 
\end{definition}
Note that if we let $h\mathrel{\mathop{:}}=g-f$, then indeed we are writing $f$ as a difference of two convex functions $f=g-h.$
It is known that any polynomial $f$ has a (polynomial) dcd $g$. A proof of this is given, e.g., in \cite{Wang}, or in Section \ref{subsec:diffconv}, where it is obtained as corollary of a stronger theorem (see Corollary \ref{cor:SDSOSetc}). By default, all dcds considered in the sequel will be of even degree. Indeed, if $f$ is of even degree $2d$, then it admits a dcd $g$ of degree $2d$. If $f$ is of odd degree $2d-1$, it can be viewed as a polynomial $\tilde{f}$ of even degree $2d$ with highest-degree coefficients which are 0. The previous result then remains true, and $\tilde{f}$ admits a dcd of degree $2d$. 

Our results show that such a decomposition can be found efficiently (e.g., by linear programming); see Theorem \ref{th:mainth}. Interestingly enough though, it is not easy to check if a candidate $g$ is a valid dcd of $f$.

\begin{proposition}\label{prop:dcd.NPhard}
Given two $n$-variate polynomials $f$ and $g$ of degree 4, with $f\neq g$, it is strongly NP-hard \footnote{For a
strongly NP-hard problem, even a pseudo-polynomial time algorithm cannot exist unless P=NP \cite{GareyJohnson_Book}.} to determine whether $g$ is a dcd of $f$.\footnote{If we do not add the condition on the input that $f\neq g$, the problem would again be NP-hard (in fact, this is even easier to prove). However, we believe that in any interesting instance of this question, one would have $f\neq g$.}
\end{proposition}
\begin{proof}
We will show this via a reduction from the problem of testing nonnegativity of biquadratic forms, which is already known to be strongly NP-hard \cite{Ling_et_al_Biquadratic}, \cite{NPhard_Convexity_MathProg}. A biquadratic form $b(x,y)$ in the variables $x=(x_1,\ldots,x_n)^T$ and $y=(y_1,\ldots,y_m)^T$ is a quartic form that can be written as $$b(x;y)=\sum_{i\leq j, k\leq l} a_{ijkl} x_ix_jy_ky_l.$$
Given a biquadratic form $b(x;y)$, define the $n \times n$ polynomial matrix $C(x,y)$ by setting $[C(x,y)]_{ij} \mathrel{\mathop{:}}=\frac{\partial b(x;y)}{\partial x_i \partial y_j},$  and let $\gamma$ be the largest coefficient in absolute value of any monomial present in some entry of $C(x,y)$. Moreover, we define 
$$r(x;y)\mathrel{\mathop{:}}=\frac{n^2 \gamma}{2}\sum_{i=1}^n x_i^4+\sum_{i=1}^{n} y_i^4+\sum_{1\leq i<j\leq n} x_i^2x_j^2+\sum_{1\leq i<j \leq n} y_i^2y_j^2.$$ 
It is proven in \cite[Theorem 3.2.]{NPhard_Convexity_MathProg} that $b(x;y)$ is nonnegative if and only if $$q(x,y) \mathrel{\mathop{:}}=b(x;y)+r(x,y)$$ is convex. We now give our reduction. Given a biquadratic form $b(x;y)$, we take 
$g=q(x,y)+r(x,y)$ and $f=r(x,y)$. If $b(x;y)$ is nonnegative, from the theorem quoted above, $g-f=q$ is convex. Furthermore, it is straightforward to establish that $r(x,y)$ is convex, which implies that $g$ is also convex. This means that $g$ is a dcd of $f$. If $b(x;y)$ is not nonnegative, then we know that $q(x,y)$ is not convex. This implies that $g-f$ is not convex, and so $g$ cannot be a dcd of $f$. \qed
\end{proof}

Unlike the quartic case, it is worth noting that in the quadratic case, it is easy to test whether a polynomial $g(x)=x^TGx$ is a dcd of $f(x)=x^TFx$. Indeed, this amounts to testing whether $F \succeq 0$ and $G -F \succeq 0$ which can be done in $O(n^3)$ time. 

As mentioned earlier, there is not only one dcd for a given polynomial $f$, but an infinite number. Indeed, if $f=g-h$ with $g$ and $h$ convex then any convex polynomial $p$ generates a new dcd $f=(g+p)-(h+p)$. It is natural then to investigate if some dcds are better than others, e.g., for use in the convex-concave procedure.

Recall that the main idea of CCP is to upperbound the non-convex function $f=g -h$ by a convex function $f^k$. These convex functions are obtained by linearizing $h$ around the optimal solution of the previous iteration. Hence, a reasonable way of choosing a good dcd would be to look for dcds of $f$ that minimize the curvature of $h$ around a point. Two natural formulations of this problem are given below. The first one attempts to minimize the \emph{average}\footnote{ Note that $\mbox{Tr } H_h(\bar{x})$ (resp. $\lambda_{\max} H_h(\bar{x})$) gives the average (resp. maximum) of $y^TH_h(\bar{x})y$ over $\{y~|~||y||=1\}$.} curvature of $h$ at a point $\bar{x}$ over all directions:
\begin{equation} \label{eq:trace.point}
\begin{aligned}
&\min_{g} \mbox{Tr } H_h(\bar{x})\\
&\text{s.t. } f=g-h, g,h \text{ convex}.
\end{aligned}
\end{equation}
The second one attempts to minimize the \emph{worst-case}\footnotemark[\value{footnote}] curvature of $h$ at a point $\bar{x}$ over all directions:
\begin{equation} \label{eq:lambda.max.point}
\begin{aligned}
&\min_{g} \lambda_{\max} H_h(\bar{x})\\
&\text{s.t. } f-g-h, g, h \text{ convex}.
\end{aligned}
\end{equation}
A few numerical experiments using these objective functions will be presented in Section \ref{subsec:scalibility}. 

Another popular notion that appears in the literature and that also relates to finding dcds with minimal curvature is that of \emph{undominated dcds.} These were studied in depth by Bomze and Locatelli in the quadratic case \cite{Bomze}. We extend their definition to general polynomials here. 

\begin{definition}
Let $g$ be a dcd of $f$. A dcd $g'$ of $f$ is said to dominate $g$ if $g-g'$ is convex and nonaffine.
A dcd $g$ of $f$ is \emph{undominated} if no dcd of $f$ dominates $g$.
\end{definition}
Arguments for chosing undominated dcds can be found in \cite{Bomze}, \cite[Section 3]{Dur}. One motivation that is relevant to CCP appears in Proposition \ref{th:CCPundom}\footnote{A variant of this proposition in the quadratic case appears in \cite[Proposition 12]{Bomze}.}. Essentially, the proposition shows that if we were to start at some initial point and apply one iteration of CCP, the iterate obtained using a dc decomposition $g$ would always beat an iterate obtained using a dcd dominated by $g$.
%
%
\begin{proposition} \label{th:CCPundom}
Let $g$ and $g'$ be two dcds of $f$. Define the convex functions $h\mathrel{\mathop{:}}=g-f$ and $h'\mathrel{\mathop{:}}=g'-f$, and assume that $g'$ dominates $g$. For a point $x_0$ in $\mathbb{R}^n$, define the convexified versions of $f$
\begin{align*}
f_g(x)\mathrel{\mathop{:}}=g(x)-(h(x_0)+\nabla h(x_0)^T(x-x_0)),\\
f_{g'}(x)\mathrel{\mathop{:}}=g'(x)-(h'(x_0)+\nabla h'(x_0)^T(x-x_0)).
\end{align*}  
Then, we have $$f_g'(x) \leq f_{g}(x), \forall x.$$
\end{proposition}

\begin{proof}
As $g'$ dominates $g$, there exists a nonaffine convex polynomial $c$ such that $c=g-g'$.
We then have $g'=g-c$ and $h'=h-c$, and
\begin{equation*}
\begin{aligned}
f_g'(x)&=g(x)-c(x)-h(x_0)+c(x_0)-\nabla h(x_0)^T(x-x_0)+\nabla c(x_0)^T(x-x_0)\\
&= f_{g}(x)-(c(x)-c(x_0)-\nabla c(x_0)^T (x-x_0)).\\
\end{aligned}
\end{equation*}
The first order characterization of convexity of $c$ then gives us  
\begin{align*}
f_{g'}(x) \leq f_{g}(x), \forall x. ~~\qed
\end{align*}
\end{proof}

In the quadratic case, it turns out that an optimal solution to (\ref{eq:trace.point}) is an undominated dcd \cite{Bomze}. A solution given by (\ref{eq:lambda.max.point}) on the other hand is not necessarily undominated. Consider the quadratic function $$f(x)=8x_1^2-2x_2^2-8x_3^2$$ and assume that we want to decompose it using (\ref{eq:lambda.max.point}). An optimal solution is given by $g^*(x)=8x_1^2+6x_2^2$ and $h^*(x)=8x_2^2+8x_3^2$ with $\lambda_{\max}H_h=8.$ This is clearly dominated by $g'(x)=8x_1^2$ as $g^*(x)-g'(x)=6x_2^2$ which is convex.

When the degree is higher than 2, it is no longer true however that solving (\ref{eq:trace.point}) returns an undominated dcd. Consider for example the degree-4 polynomial $$f(x)=x^{12}-x^{10}+x^{6}-x^{4}.$$
A solution to (\ref{eq:trace.point}) with $\bar{x}=0$ is given by $g(x)=x^{12}+x^{6}$ and $h(x)=x^{10}+x^{4}$ (as $\mbox{Tr}H_h(0)=0$). This is dominated by the dcd $g(x)=x^{12}-x^8+x^6$ and $h(x)=x^{10}-x^8+x^4$ as $g-g'=x^8$ is clearly convex.

It is unclear at this point how one can obtain an undominated dcd for higher degree polynomials, or even if one exists. In the next theorem, we show that such a dcd always exists and provide an optimization problem whose optimal solution(s) will always be undominated dcds. This optimization problem involves the integral of a polynomial over a sphere which conveniently turns out to be an explicit linear expression in its coefficients.

\begin{proposition}[\cite{intSphere}]
Let $S^{n-1}$ denote the unit sphere in $\mathbb{R}^n$. For a monomial $x_1^{\alpha_1}\ldots x_n^{\alpha_n}$, define $\beta_j\mathrel{\mathop{:}}=\frac{1}{2} (\alpha_j+1)$.
Then $$\int_{S^{n-1}} x_1^{\alpha_1}\ldots x_n^{\alpha_n} d\sigma = \begin{cases} 0 &\text{ if some $\alpha_j$ is odd}, \\ \frac{2\Gamma(\beta_1) \ldots \Gamma(\beta_n)}{\Gamma(\beta_1+\ldots+\beta_n)} &\text{ if all $\alpha_j$ are even,} \end{cases}$$
\end{proposition}
where $\Gamma$ denotes the gamma function, and $\sigma$ is the rotation invariant probability measure on $S^{n-1}.$



\begin{theorem}\label{thm:undom.dcd}
Let $f \in \tilde{\mathcal{H}}_{n,2d}.$  Consider the optimization problem
\begin{equation} \label{eq:undom.dcd}
\begin{aligned}
\min_{g \in \tilde{\mathcal{H}}_{n,2d}} &\frac{1}{\mathcal{A}_n}\int_{S^{n-1}} \mbox{Tr } H_g d \sigma\\
\text{s.t. } &g \text{ convex}, \\
&g-f \text{ convex},
\end{aligned}
\end{equation}
where $\mathcal{A}_n=\frac{2\pi^{n/2}}{\Gamma(n/2)}$ is a normalization constant which equals the area of $S^{n-1}$. Then, an optimal solution to (\ref{eq:undom.dcd}) exists and any optimal solution is an undominated dcd of $f$.
\end{theorem}
Note that problem (\ref{eq:undom.dcd}) is exactly equivalent to (\ref{eq:trace.point}) in the case where $n=2$ and so can be seen as a generalization of the quadratic case.
\begin{proof}
We first show that an optimal solution to (\ref{eq:undom.dcd}) exists. As any polynomial $f$ admits a dcd, (\ref{eq:undom.dcd}) is feasible. Let $\tilde{g}$ be a dcd of $f$ and define $\gamma \mathrel{\mathop{:}}= \int_{S^{n-1}} \mbox{Tr }H_{\tilde{g}} d\sigma.$
 Consider the optimization problem given by (\ref{eq:undom.dcd}) with the additional constraints:
\begin{equation} \label{eq:undom.dcd.new.constraint}
\begin{aligned}
\min_{g \in \tilde{\mathcal{H}}_{n,2d}} &\frac{1}{\mathcal{A}_n}\int_{S^{n-1}} \mbox{Tr } H_g d \sigma\\
\text{s.t. } &g \text{ convex and with no affine terms} \\
&g-f \text{ convex},\\
&\int_{S^{n-1}} \mbox{Tr } H_{g} d\sigma \leq \gamma.
\end{aligned}
\end{equation}
Notice that any optimal solution to (\ref{eq:undom.dcd.new.constraint}) is an optimal solution to (\ref{eq:undom.dcd}). Hence, it suffices to show that (\ref{eq:undom.dcd.new.constraint}) has an optimal solution. Let $\mathcal{U}$ denote the feasible set of (\ref{eq:undom.dcd.new.constraint}). Evidently, the set $\mathcal{U}$ is closed and $g \rightarrow \int_{S^{n-1}} \mbox{Tr }H_g d \sigma$ is continuous. If we also show that $\mathcal{U}$ is bounded, we will know that the optimal solution to (\ref{eq:undom.dcd.new.constraint}) is achieved.
To see this, assume that $\mathcal{U}$ is unbounded. Then for any $\beta$,  there exists a coefficient $c_g$ of some $g \in \mathcal{U}$ that is larger than $\beta$. By absence of affine terms in $g$, $c_g$ features in an entry of $H_g$ as the coefficient of a nonzero monomial. Take $\bar{x} \in S^{n-1}$ such that this monomial evaluated at $\bar{x}$ is nonzero: this entails that at least one entry of $H_g(\bar{x})$ can get arbitrarily large. However, since $g \rightarrow \mbox{Tr }H_g$ is continuous and $\int_{S^{n-1}} \mbox{Tr }H_g d\sigma \leq \gamma$, $\exists \bar{\gamma}$ such that $\mbox{Tr }H_g(x) \leq \bar{\gamma}$, $\forall x \in S^{n-1}$. This, combined with the fact that $H_g(x) \succeq 0~ \forall x$, implies that $||H_g(x)|| \leq \bar{\gamma}, ~\forall x \in S^{n-1}$, which contradicts the fact that an entry of $H_g(\bar{x})$ can get arbitrarily large.


We now show that if $g^*$ is any optimal solution to (\ref{eq:undom.dcd}), then $g^*$ is an undominated dcd of $f$. Suppose that this is not the case. Then, there exists a dcd $g'$ of $f$ such that $g^*-g'$ is nonaffine and convex. As $g'$ is a dcd of $f$, $g'$ is feasible for (\ref{eq:undom.dcd}). The fact that $g^*-g'$ is nonaffine and convex implies that $$\int_{S^{n-1}} \mbox{Tr } H_{g^*-g'} d\sigma >0 \Leftrightarrow \int_{S^{n-1}} \mbox{Tr} H_{g^*} d \sigma >\int_{S^{n-1}} \mbox{Tr } H_{g'} d\sigma,$$  which contradicts the assumption that $g^*$ is optimal to (\ref{eq:undom.dcd}). \qed
\end{proof}

Although optimization problem (\ref{eq:undom.dcd}) is guaranteed to produce an undominated dcd, we show that unfortunately it is intractable to solve.

\begin{proposition}\label{prop:undominated.NPhard}
Given an n-variate polynomial $f$ of degree 4 with rational coefficients, and a rational number $k$, it is strongly NP-hard to decide whether there exists a feasible solution to (\ref{eq:undom.dcd}) with objective value $\leq k$.
\end{proposition}
\begin{proof}
We give a reduction from the problem of deciding convexity of quartic polynomials. Let $q$ be a quartic polynomial. We take $f=q$ and $k=\frac{1}{\mathcal{A}_n} \int_{S^{n-1}} \mbox{Tr } H_q(x)$. If $q$ is convex, then $g=q$ is trivially a dcd of $f$ and
\begin{align} \label{eq:condition.g}
\frac{1}{\mathcal{A}_n} \int_{S^{n-1}} \mbox{Tr } H_g d\sigma \leq k.
\end{align} 
If $q$ is not convex, assume that there exists a feasible solution $g$ for (\ref{eq:undom.dcd}) that satisfies (\ref{eq:condition.g}). From (\ref{eq:condition.g}) we have
\begin{align} \label{eq:ineq1}
\int_{S^{n-1}} \mbox{Tr } H_g(x) \leq \int_{S^{n-1}} \mbox{Tr } H_f d\sigma \Leftrightarrow \int_{S^{n-1}} \mbox{Tr } H_{f-g} d \sigma \geq 0.
\end{align}
But from (\ref{eq:undom.dcd}), as $g-f$ is convex, $\int_{S^{n-1}} \mbox{Tr }H_{g-f} d\sigma \geq 0.$
Together with (\ref{eq:ineq1}), this implies that $$\int_{S^{n-1}} \mbox{Tr }H_{g-f}d \sigma=0$$ which in turn {\gh implies that $H_{g-f}(x)=H_g(x)-H_f(x)=0$. To see this, note that $\mbox{Tr}(H_{g-f})$ is a nonnegative polynomial which must be identically equal to $0$ since its integral over the sphere is $0$. As $H_{g-f}(x)\succeq 0,\forall x$, we get that $H_{g-f}=0.$ Thus, $H_g(x)=H_f(x), ~\forall x$, which is not possible as $g$ is convex and $f$ is not.}\qed
\end{proof}
We remark that solving (\ref{eq:undom.dcd}) in the quadratic case (i.e., $2d=2$) is simply a semidefinite program.

\section{Alegbraic relaxations and more tractable subsets of the set of convex polynomials} \label{sec:ConvRelax}

We have just seen in the previous section that for polynomials with degree as low as four, some basic tasks related to dc decomposition are computationally intractable. In this section, we identify three subsets of the set of convex polynomials that lend themselves to polynomial-time algorithms. These are the sets of \emph{sos-convex, sdsos-convex}, and \emph{dsos-convex} polynomials, which will respectively lead to semidefinite, second order cone, and linear programs. The latter two concepts are to our knowledge new and are meant to serve as more scalable alternatives to sos-convexity. All three concepts certify convexity of polynomials via explicit algebraic identities, which is the reason why we refer to them as algebraic relaxations.


\subsection{DSOS-convexity, SDSOS-convexity, SOS-convexity}

To present these three notions we need to introduce some notation and briefly review the concepts of sos, dsos, and sdsos polynomials.

We denote the set of nonnegative polynomials (resp. forms) in $n$ variables and of degree $d$ by $\tilde{PSD}_{n,d}$ (resp. $PSD_{n,d}$).
A polynomial $p$ is a \emph{sum of squares} (sos) if it can be written as $p(x)=\sum_{i=1}^r q_i^2(x)$ for some polynomials $q_1,\ldots,q_r$. The set of sos polynomials (resp. forms) in $n$ variables and of degree $d$ is denoted by $\tilde{SOS}_{n,d}$ (resp. $SOS_{n,d}$).
We have the obvious inclusion $\tilde{SOS}_{n,d}\subseteq \tilde{PSD}_{n,d}$ (resp. $SOS_{n,d}\subseteq PSD_{n,d}$), which is strict unless  $d=2$, or $n=1$, or $(n,d)=(2,4)$ (resp. $d=2$, or $n=2$, or $(n,d)=(3,4)$)~\cite{Hilbert_1888},~\cite{Reznick}.

Let $\tilde{z}_{n,d}(x)$ (resp. $z_{n,d}(x)$) denote the vector of all monomials in $x=(x_1,\ldots,x_n)$ of degree up to (resp. exactly) $d$; the length of this vector is $\binom{n+d}{d}$ (resp. $\binom{n+d-1}{d}$). It is well known that a polynomial (resp. form) $p$ of degree $2d$ is sos if and only if it can be written as $p(x)=\tilde{z}^T_{n,d}(x)Q\tilde{z}_{n,d}(x)$ (resp. $p(x)=z_{n,d}^T(x)Qz_{n,d}(x)$), for some psd matrix $Q$~\cite{sdprelax},~\cite{PhD:Parrilo}. The matrix $Q$ is generally called the Gram matrix of $p$. An SOS optimization problem is the problem of minimizing a linear function over the intersection of the convex cone $SOS_{n,d}$ with an affine subspace. The previous statement implies that SOS optimization problems can be cast as semidefinite programs.

We now define dsos and sdsos polynomials, which were recently proposed by Ahmadi and Majumdar~\cite{iSOS_journal},~\cite{dsos_ciss14} as more tractable subsets of sos polynomials. When working with dc decompositions of $n$-variate polynomials, we will end up needing to impose sum of squares conditions on polynomials that have $2n$ variables (see Definition~\ref{def:alternatives.convexity}). While in theory the SDPs arising from sos conditions are of polynomial size, in practice we rather quickly face a scalability challenge. For this reason, we also consider the class of dsos and sdsos polynomials, which while more restrictive than sos polynomials, are considerably more tractable. For example, Table \ref{tab:dsos} in Section \ref{subsec:scalibility} shows that when $n=14$, dc decompositions using these concepts are about 250 times faster than an sos-based approach. At $n=18$ variables, we are unable to run the sos-based approach on our machine. With this motivation in mind, let us start by recalling some concepts from linear algebra.



\begin{definition}\label{def:dd.sdd}
A symmetric matrix $M$ is said to be \emph{diagonally dominant (dd)} if $m_{ii} \geq \sum_{j \neq i}|m_{ij}|$ for all $i$, and \emph{strictly diagonally dominant} if  $m_{ii} > \sum_{j \neq i}|m_{ij}|$ for all $i$. 
We say that $M$ is \emph{scaled diagonally dominant (sdd)} if there exists a diagonal matrix $D,$ with positive diagonal entries, such that $DAD$ is dd.
\end{definition}
We have the following implications from Gershgorin's circle theorem
\begin{align}\label{eq:implic.matrices}
\text{M } dd \Rightarrow \text{M } sdd \Rightarrow \text{M } psd.
\end{align}
Furthermore, notice that requiring $M$ to be dd can be encoded via a linear program (LP) as the constraints are linear inequalities in the coefficients of $M$. Requiring that $M$ be sdd can be encoded via a second order cone program (SOCP). This follows from the fact that $M$ is sdd if and only if $$M=\sum_{i<j} M^{ij}_{2\times 2},$$ where each $M^{ij}_{2 \times 2}$ is an $n \times n$ symmetric matrix with zeros everywhere except four entries $M_{ii},M_{ij},M_{ji}, M_{jj}$ which must make the $2 \times 2$ matrix $\begin{pmatrix} M_{ii} & M_{ij} \\ M_{ji} & M_{jj} \end{pmatrix}$ symmetric positive semidefinite \cite{iSOS_journal}. These constraints are \emph{rotated quadratic cone constraints} and can be imposed via SOCP \cite{Alizadeh}.

\begin{definition}[\cite{iSOS_journal}]
A polynomial $p \in \tilde{\mathcal{H}}_{n,2d}$ is said to be 
\begin{itemize}
\item \emph{diagonally-dominant-sum-of-squares (dsos)} if it admits a representation $p(x)=\tilde{z}^T_{n,d}(x)Q\tilde{z}_{n,d}(x)$, where $Q$ is a dd matrix.
\item \emph{scaled-diagonally-dominant-sum-of-squares (sdsos)} it it admits a representation $p(x)=\tilde{z}^T_{n,d}(x)Q\tilde{z}_{n,d}(x),$ where $Q$ is an sdd matrix.
\end{itemize}
Identical conditions involving $z_{n,d}$ instead of $\tilde{z}_{n,d}$ define the sets of dsos and sdsos forms. 
\end{definition}

The following implications are again straightforward:
\begin{align}\label{eq:implicsdsos}
p(x) \text{ dsos} \Rightarrow p(x) \text{ sdsos} \Rightarrow p(x) \text{ sos} \Rightarrow p(x) \text{ nonnegative}.
\end{align}
Given the fact that our Gram matrices and  polynomials are related to each other via linear equalities, it should be clear that optimizing over the set of dsos (resp. sdsos, sos) polynomials is an LP (resp. SOCP, SDP). 

Let us now get back to convexity.

\begin{definition} \label{def:alternatives.convexity}
Let $y=(y_1,\ldots,y_n)^T$ be a vector of variables. A polynomial $p\mathrel{\mathop:}=p(x)$ is said to be
\begin{itemize}
\item \emph{dsos-convex} if $y^TH_p(x)y$ is dsos {\gh (as a polynomial in $x$ and $y$)}.
\item \emph{sdsos-convex} if $y^TH_p(x)y$ is sdsos {\gh (as a polynomial in $x$ and $y$)}.
\item \emph{sos-convex} if $y^TH_p(x)y$ is sos {\gh (as a polynomial in $x$ and $y$)}.\footnote{The notion of sos-convexity has already appeared in the study of semidefinite representability of convex sets~\cite{helton2010} and in  applications such as shaped-constrained regression in statistics~\cite{convex_fitting}.}
\end{itemize}

We denote the set of dsos-convex (resp. sdsos-convex, sos-convex, convex) forms in $\mathcal{H}_{n,2d}$ by $\Sigma_DC_{n,2d}$ (resp. $\Sigma_SC_{n,2d}$, $\Sigma C_{n,2d}$, $C_{n,2d}$). Similarly, $\tilde{\Sigma}_DC_{n,2d}$ (resp. $\tilde{\Sigma}_SC_{n,2d}$, $\tilde{\Sigma} C_{n,2d}$, $\tilde{C}_{n,2d}$) denote the set of dsos-convex (resp. sdsos-convex, sos-convex, convex) polynomials in $\tilde{\mathcal{H}}_{n,2d}$.
\end{definition}
The following inclusions 
\begin{align}\label{eq:inclusion.cones}
\Sigma_DC_{n,2d} \subseteq \Sigma_SC_{n,2d} \subseteq \Sigma C_{n,2d} \subseteq C_{n,2d}
\end{align}
are a direct consequence of (\ref{eq:implicsdsos}) and the second-order necessary and sufficient condition for convexity which reads $$ p(x) \text{ is convex } \Leftrightarrow H_p(x) \succeq 0, \forall x \in \mathbb{R}^n \Leftrightarrow y^TH_p(x)y \geq 0, \forall x,y \in \mathbb{R}^n.$$ 
Optimizing over $\Sigma_DC_{n,2d}$ (resp. $\Sigma_S C_{n,2d}$, $\Sigma C_{n,2d}$) is an LP (resp. SOCP, SDP). The same statements are true for $\tilde{\Sigma}_DC_{n,2d}$, $\tilde{\Sigma}_SC_{n,2d}$ and $\tilde{\Sigma} C_{n,2d}$.

Let us draw these sets for a parametric family of polynomials
\begin{align} \label{eq:parampoly}
p(x_1,x_2)=2x_1^4+2x_2^4+ax_1^3x_2+bx_1^2x_2^2+cx_1x_2^3.
\end{align} Here, $a,b$ and $c$ are parameters. It is known that for bivariate quartics, all convex polynomials are sos-convex; i.e., $\Sigma C_{2,4}=C_{2,4}.$\footnote{In general, constructing polynomials that are convex but not sos-convex seems to be a nontrivial task~\cite{AAA_PP_not_sos_convex_journal}. A complete characterization of the dimensions and degrees for which convexity and sos-convexity are equivalent is given in~\cite{AAA_PP_table_sos-convexity}.}  To obtain Figure \ref{fig:cones}, we fix $c$ to some value and then plot the values of $a$ and $b$ for which $p(x_1,x_2)$ is s/d/sos-convex. As we can see, the quality of the inner approximation of the set of convex polynomials by the sets of dsos/sdsos-convex polynomials can be very good (e.g., $c=0$) or less so (e.g., $c=1$).

\begin{figure}[h!]
\centering
\includegraphics[scale=0.38]{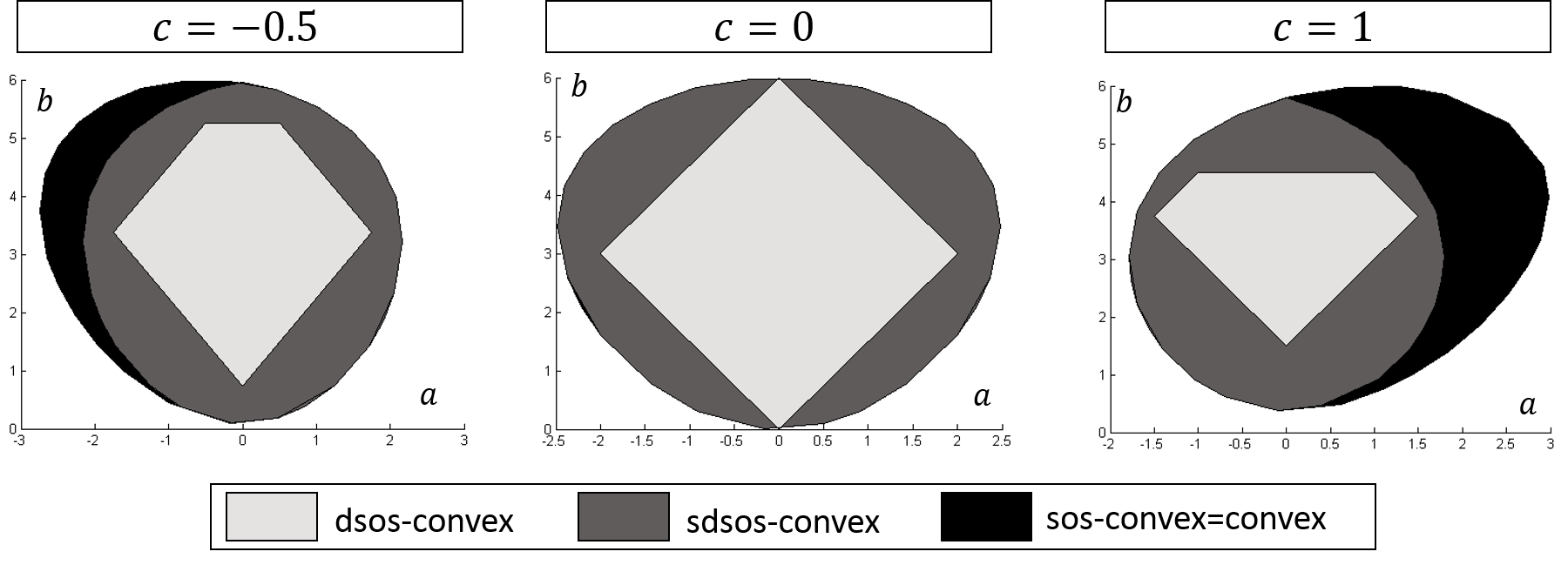}
\caption{The sets $\Sigma_DC_{n,2d}, \Sigma_SC_{n,2d}$ and $\Sigma C_{n,2d}$ for the parametric family of polynomials in (\ref{eq:parampoly})}
\label{fig:cones}
\end{figure}

\subsection{Existence of difference of s/d/sos-convex decompositions of polynomials} \label{subsec:diffconv}
%
%
%
%
%
%

 The reason we introduced the notions of s/d/sos-convexity is that in our optimization problems for finding dcds, we would like to replace the condition $$f=g-h,~ g,h \text{ convex}$$ with the computationally tractable condition $$f=g-h, ~g,h \text{ s/d/sos-convex.}$$ 
The first question that needs to be addressed is whether for any polynomial such a decomposition exists. In this section, we prove that this is indeed the case. This in particular implies that a dcd can be found efficiently. 


We start by proving a lemma about cones. 

\begin{lemma}\label{lemma:cones}
Consider a vector space $E$ and a full-dimensional cone $K \subseteq E$. Then, any $v \in E$ can be written as $v=k_1-k_2,$ where $k_1,k_2 \in K.$
\end{lemma}

\begin{proof}
Let $v \in E$. If $v \in K$, then we take $k_1=v$ and $k_2=0$. Assume now that $v \notin K$ and let $k$ be any element in the interior of the cone $K$. As $k \in int(K)$, there exists $0<\alpha<1$ such that $k'\mathrel{\mathop{:}}=(1-\alpha)v+\alpha k \in K.$ Rewriting the previous equation, we obtain $$v=\frac{1}{1-\alpha}k'-\frac{\alpha}{1-\alpha}k.$$ By taking $k_1\mathrel{\mathop{:}}=\frac{1}{1-\alpha}k'$ and $k_2\mathrel{\mathop{:}}=\frac{\alpha}{1-\alpha}k$, we observe that $v=k_1-k_2$ and $k_1,k_2 \in K$. \qed
\end{proof}

The following theorem is the main result of the section. 
\begin{theorem}\label{thm:diff.dsos}
Any polynomal $p \in \tilde{\mathcal{H}}_{n,2d}$ can be written as the difference of two dsos-convex polynomials in $\tilde{\mathcal{H}}_{n,2d}.$
\end{theorem}

\begin{corollary}\label{cor:SDSOSetc}
Any polynomial $p \in \tilde{\mathcal{H}}_{n,2d}$ can be written as the difference of two sdsos-convex, sos-convex, or convex polynomials in $\tilde{\mathcal{H}}_{n,2d}$.
\end{corollary}
\begin{proof}
This is straightforward from the inclusions $$\tilde{\Sigma}_DC_{n,2d} \subseteq \tilde{\Sigma}_SC_{n,2d} \subseteq \tilde{\Sigma}C_{n,2d} \subseteq \tilde{C}_{n,2d}.\ \ \qed$$
\end{proof}
In view of Lemma \ref{lemma:cones}, it suffices to show that $\tilde{\Sigma}_DC_{n,2d}$ is full dimensional in the vector space $\tilde{\mathcal{H}}_{n,2d}$ to prove Theorem \ref{thm:diff.dsos}. We do this by constructing a polynomial in $int(\tilde{\Sigma}_DC_{n,2d})$ for any $n,d$.


Recall that $z_{n,d}$ (resp. $\tilde{z}_{n,d}$) denotes the vector of all monomials in $x=(x_1,\ldots,x_n)$ of degree exactly (resp. up to) $d$. If $y=(y_1,\ldots,y_n)$ is a vector of variables of length $n$, we define $$w_{n,d}(x,y) \mathrel{\mathop{:}}= y \cdot z_{n,d}(x),$$
where $y \cdot z_{n,d}(x)=(y_1 z_{n,d}(x), \ldots, y_n  z_{n,d}(x))^T.$ 
%
%
Analogously, we define $$\tilde{w}_{n,d}(x,y)\mathrel{\mathop{:}}=y \cdot \tilde{z}_{n,d}(x).$$

\begin{theorem} \label{th:mainth}
For all $n,d$, there exists a polynomial $p \in \tilde{\mathcal{H}}_{n,2d}$ such that 
\begin{align}\label{eq:mainth}
y^TH_p(x)y=\tilde{w}^T_{n,d-1}(x,y)Q\tilde{w}_{n,d-1}(x,y),
\end{align} where $Q$ is strictly dd.
\end{theorem}
Any such polynomial will be in $int(\tilde{\Sigma}_D C_{n,2d})$. Indeed, if we were to pertub the coefficients of $p$ slightly, then each coefficient of $Q$ would undergo a slight perturbation. As $Q$ is strictly dd, $Q$ would remain dd, and hence $p$ would remain dsos-convex. 

We will prove Theorem \ref{th:mainth} through a series of lemmas. First, we show that this is true in the homogeneous case and when $n=2$ (Lemma \ref{lemma:homogn2}). By induction, we prove that this result still holds in the homogeneous case for any $n$ (Lemma \ref{lemma:induction}). We then extend this result to the nonhomogeneous case.

\begin{lemma} \label{lemma:homogn2}
For all $d$, there exists a polynomial $p \in \tilde{\mathcal{H}}_{2,2d}$ such that 
\begin{align}\label{eq:DSOS.Conv.int}
y^TH_p(x)y=w^T_{2,d-1}(x,y)Qw_{2,d-1}(x,y), 
\end{align}
for some strictly dd matrix $Q$.
\end{lemma}

{\gh We remind the reader that Lemma \ref{lemma:homogn2} corresponds to the base case of a proof by induction on $n$ for Theorem \ref{th:mainth}.}

\begin{proof}
{\gh In this proof, we show that there exists a polynomial $p$ that satisfies (\ref{eq:DSOS.Conv.int}) for some strictly dd matrix $Q$ in the case where $n=2$, and for any $d\geq 1.$}

First, if $2d=2$, we simply take $p(x_1,x_2)=x_1^2+x_2^2$ as $y^TH_p(x)y=2y^T I y$ {\gh and the identity matrix is strictly dd}. Now, assume $2d>2$.  We consider two cases depending on whether $d$ is divisible by $2$.

In the case that it is, we construct $p$ as
\scalefont{0.7}
\begin{align*}
p(x_1,x_2) \mathrel{\mathop{:}}=a_0x_1^{2d}+a_1x_1^{2d-2}x_2^2+a_2x_1^{2d-4}x_2^4+\ldots+a_{d/2}x_1^{d}x_2^{d}+\ldots+a_1x_1^2x_2^{2d-2}+a_0x_2^{2d},
\end{align*}
\normalsize
with the sequence $\{a_k\}_{k=0,\ldots,\frac{d}{2}}$ defined as follows
\begin{equation}\label{eq:ak.even}
\begin{aligned}
a_1&=1\\
a_{k+1}&=\left( \frac{2d-2k}{2k+2}\right)a_k,~ k=1,\ldots, \frac{d}{2}-1 \text{ (for } 2d>4)\\
a_0&=\frac{1}{d}+\frac{d}{2(2d-1)}a_{\frac{d}{2}}.
\end{aligned}
\end{equation}


Let 
\begin{equation}\label{eq:def.beta.gamma.delta}
\begin{aligned}
\beta_k&=a_k(2d-2k)(2d-2k-1), k=0,\ldots,\frac{d}{2}-1,\\
\gamma_k&=a_k\cdot 2k(2k-1), k=1,\ldots,\frac{d}{2},\\
\delta_k&=a_k(2d-2k)\cdot 2k, k=1,\ldots,\frac{d}{2}.
\end{aligned}
\end{equation}
We claim that the matrix $Q$ defined as  
\setcounter{MaxMatrixCols}{29}
\begin{align*}
\scalefont{0.65}
\begin{pmatrix}
\beta_0 & & & & & &  &  & &  & & \vline & 0 & \delta_1 & & & & & & & & &\delta_{\frac{d}{2}}\\
& \ddots & & & & &  & & & & & \vline &  & \ddots & \ddots & &  & & & & & &\\
 &  & \beta_{k} & & & & & & & & &\vline & & & 0 & \delta_{k+1}   & & & & & &\\
 & & & \ddots& & & & & & & & \vline & & & & \ddots & \ddots& & & & &\\
 & & & & \beta_{\frac{d}{2}-2} & & & & & & & \vline & & & & & \ddots& \delta_{\frac{d}{2}-1} & & & &\\
 & & & & & \beta_{\frac{d}{2}-1} & & & & & &\vline & & &  & & &0 & 0& & & &\\
\hline
 &  & & & & &\gamma_{\frac{d}{2}} & & & & &\vline & &   & & & & & 0 & \delta_{\frac{d}{2}-1} & & & \\
 &  & & & & & & \ddots & & & & \vline & & &  & & & & &\ddots  &\ddots& & \\
  &  & & & & & & & \gamma_k&  & &\vline & & &  & & & & &  & 0 &\delta_{k-1} & \\
 &  & & &  & &  & &  &  \ddots & &\vline & & &  & & & & & & &\ddots & \ddots \\
 &  & &   & & &   & &  & & \gamma_1 &\vline & & &  & & & & & &  & & 0 \\
\hline \hline
0& & & & & & & & &  & & \vline & \gamma_1 &  & & & & & & & & & &\\
\ddots & \ddots & & & & & & & & & & \vline &  &\ddots  & & & & & & & & & &\\
& \delta_{k-1}& 0  & & & & & & & &  & \vline & &  &\gamma_k  & & & & & & & & & &\\
& & \ddots & \ddots & & & & & &  & &\vline & &  & & \ddots & & & & & & & &\\
& &  & \delta_{\frac{d}{2}-1}& 0 & & & & & & & \vline & & &  & &   \gamma_{\frac{d}{2}} & & & & & \\
\hline & &  &  &0 &0 & & & &  & & \vline & &  & & &  &   \beta_{\frac{d}{2}-1} & & & & & &\\
 & &  &  & &\delta_{\frac{d}{2}-1} & 0 & & & & & \vline & &  & & &  &   & \beta_{\frac{d}{2}-2} & & & & & &\\
& &  & & & &\ddots & \ddots & & &  & \vline & &  & & &  & & &  \ddots & & & & &\\
& &  & & & & & \delta_{k+1} &0 & & & \vline & &  & & &  & &  & &\beta_k & & & &\\
& &  &  & & & &  &\ddots &\ddots &  &\vline & &  & & &  &  & & & &\ddots & & &\\
\delta_{\frac{d}{2}}& &  &  & & &  & & &  \delta_1&  0& \vline & &  & & & &  &  & & & &\beta_0 & &\\
\end{pmatrix}
\normalsize
\end{align*}
is strictly dd and satisfies (\ref{eq:DSOS.Conv.int}) with $w_{2,d-1}(x,y)$ ordered as $$\begin{pmatrix}y_1x_1^{d-1}, y_1 x_1^{d-2}x_2, \ldots, y_1x_1x_2^{d-2}, y_1x_2^{d-1},y_2x_1^{d-1}, y_2 x_1^{d-2}x_2, \ldots, y_2 x_1x_2^{d-2}, y_2x_2^{d-1} \end{pmatrix}^T.$$
To show (\ref{eq:DSOS.Conv.int}), one can derive the Hessian of $p$, expand both sides of the equation, and verify equality. To ensure that the matrix is strictly dd, we want all diagonal coefficients to be strictly greater than the sum of the elements on the row. This translates to the following inequalities
\begin{align*}
\beta_0 &>\delta_1+\delta_{\frac{d}{2}}\\
\beta_k &>\delta_{k+1}, \forall k=1,\ldots,\frac{d}{2}-2\\
\beta_{\frac{d}{2}-1}&>0, \gamma_1>0 \\
\gamma_{k+1} &>\delta_{k}, \forall k=1,\ldots, \frac{d}{2}-1. \\
\end{align*}
Replacing the expressions of $\beta_k,\gamma_k$ and $\delta_k$ in the previous inequalities using (\ref{eq:def.beta.gamma.delta}) and the values of $a_k$ given in (\ref{eq:ak.even}), one can easily check that these inequalities are satisfied.

We now consider the case where $d$ is not divisable by 2 and take
\scalefont{0.60}
\begin{align*}
p(x_1,x_2)\mathrel{\mathop{:}}=a_0x_1^{2d}+a_1x_1^{2d-2}x_2^2+\ldots+a_{(d-1)/2}x_1^{d+1}x_2^{d-1}+a_{(d-1)/2}x_1^{d-1}x_2^{d+1}
+\ldots+a_1x_1^2x_2^{2d-2}+a_0x_2^{2d},
\end{align*}
\normalsize
with the sequence $\{a_k\}_{k=0,\ldots,\frac{d-1}{2}}$ defined as follows
\begin{equation}\label{eq:ak.odd}
\begin{aligned}
a_1&=1\\
a_{k+1}&=\left( \frac{2d-2k}{2k+2}\right)a_k,~ k=1,\ldots, \frac{d-3}{2}\\
a_0&=1+\frac{2(2d-2)}{2d(2d-1)}.
\end{aligned}
\end{equation}
Again, we want to show existence of a strictly dd matrix $Q$ that satisfies (\ref{eq:DSOS.Conv.int}). Without changing the definitions of the sequences $\{\beta_k\}_{k=1,\ldots,\frac{d-3}{2}}$,$\{\gamma_k\}_{k=1,\ldots,\frac{d-1}{2}}$ and $\{\delta_k\}_{k=1,\ldots,\frac{d-1}{2}}$, we claim this time that the matrix $Q$ defined as
\begin{align*}
\scalefont{0.65}
\begin{pmatrix}
\beta_0 & & & & & &  &  & &  & \vline & 0 & \delta_1 & & & & & & & &\textbf{0}\\
& \ddots & & & & &  & & & & \vline &  & \ddots & \ddots & &  & & & & & &\\
 &  & \beta_{k} & & & & & & & &\vline & & & 0 & \delta_{k+1}   & & & & & &\\
 & & & \ddots& & & & & & & \vline & & & &\ddots & \ddots& & & & &\\
 & & & &\beta_{\frac{{d-3}}{2}} & & & & & &\vline & &  & & &0 & \delta_{\frac{{d-1}}{{2}}}& & & &\\
\hline
 &  & & & &\gamma_{\frac{{d-1}}{2}} & & & & &\vline & &   & & & & 0 & \delta_{\frac{{d-1}}{2}-1} & & & \\
 &  & & & & & \ddots & & & & \vline & & &  & & &  &\ddots  &\ddots& & \\
  &  & & & & & & \gamma_k&  & &\vline & & &  & & & &  & 0 &\delta_{k-1} & \\
 &  & & &  & & &  &  \ddots & &\vline & & &  & & & & & &\ddots & \ddots \\
 &  & &   & & &   &  & & \gamma_1 &\vline & & &  & & & & &  & & 0 \\
\hline \hline
0& & & & & & & & &  & \vline & \gamma_1 &  & & & & & & & & & &\\
\ddots & \ddots & & & & & & & &  & \vline &  &\ddots  & & & & & & & & & &\\
& \delta_{k-1}& 0  & & & & & & &  & \vline & &  &\gamma_k  & & & & & & & & & &\\
& & \ddots & \ddots & & & & & &  & \vline & &  & & \ddots & & & & & & & &\\
& &  & \delta_{\frac{d-1}{2}-1}& 0 & & & & & & \vline & & &  & &   \gamma_{\frac{d-1}{2}} & & & & & \\
\hline & &  &  &  \delta_{\frac{{d-1}}{\textbf{2}}}&0 & & & &  & \vline & &  & & &  &   \beta_{\frac{{d-3}}{2}} & & & & & &\\
& &  &  &&\ddots & \ddots & & &  & \vline & &  & & &  &  &\ddots & & & & &\\
& &  &  & & & \delta_{k+1} &0 & &  & \vline & &  & & &  &  & &\beta_k & & & &\\
& &  &  & & &  &\ddots &\ddots &  & \vline & &  & & &  &  & & &\ddots & & &\\
\textbf{0}& &  &  & & &  & &  \delta_1&  0& \vline & &  & & &  &  & & & &\beta_0 & &\\
\end{pmatrix}
\normalsize
\end{align*}
satisfies (\ref{eq:DSOS.Conv.int}) and is strictly dd. Showing (\ref{eq:DSOS.Conv.int}) amounts to deriving the Hessian of $p$ and checking that the equality is verified. To ensure that $Q$ is strictly dd, the inequalities that now must be verified are 
\begin{align*}
\beta_k &>\delta_{k+1}, \forall k=0,\ldots,\frac{d-1}{2}-1\\
\gamma_{k} &>\delta_{k-1}, \forall k=2,\ldots, \frac{d-1}{2} \\
\gamma_{1}&>0.
\end{align*}
These inequalities can all be shown to hold using (\ref{eq:ak.odd}). \qed

\end{proof}

\begin{lemma}\label{lemma:induction}
For all $n,d,$ there exists a form $p_{n,2d} \in \mathcal{H}_{n,2d}$ such that  
\begin{align*}
y^TH_{p_{n,2d}}(x)y=w^T_{n,d-1}(x,y)Q_{p_{n,2d}}w_{n,d-1}(x,y) 
\end{align*}
 and $Q_{p_{n,2d}}$ is a strictly dd matrix.
\end{lemma}
\begin{proof}
We proceed by induction on $n$ with fixed and arbitrary $d$. The property is verified for $n=2$ by Lemma \ref{lemma:homogn2}. 
Suppose that there exists a form $p_{n,2d} \in \mathcal{H}_{n,2d}$ such that 
\begin{align}\label{eq:induction.hyp}
y^TH_{p_{n,2d}}y=w^T_{n,d-1}(x,y)Q_{p_{n,2d}}w_{n,d-1}(x,y),
\end{align}
for some strictly dd matrix $Q_{p_{n,2d}}.$  We now show that 
$$p_{n+1,2d}\mathrel{\mathop{:}}=q+\alpha v $$
with
\begin{equation} \label{eq:defqv}
\begin{aligned}
q&\mathrel{\mathop{:}}=\sum_{\{i_1,\ldots,i_n\} \in \{1,\ldots,n+1\}^n} p_{n,2d}(x_{i_1},\ldots,x_{i_n})\\
v &\mathrel{\mathop{:}}=\sum_{2i_1+\ldots 2i_{n+1}=2d, i_1,\ldots,i_{n+1}>0} x_1^{2i_1}x_2^{2i_2}\ldots x_{n+1}^{2i_{n+1}},
\end{aligned}
\end{equation}
and $\alpha>0$ small enough, verifies 
\begin{align} \label{eq:induction.result}
y^TH_{p_{n+1,2d}}y=w^T_{n+1,d-1}(x,y)Q_{p_{n+1,2d}}w_{n+1,d-1}(x,y),
\end{align}
for some strictly dd matrix $Q_{p_{n+1,2d}}$. Equation (\ref{eq:induction.result}) will actually be proved using an equivalent formulation that we describe now. {\gh Recall that $$w_{n+1,d-1}(x,y)=y\cdot z_{n+1,d-1},$$ where $z_{n+1,d-1}$ is the standard vector of monomials in $x=(x_1,\ldots,x_{n+1})$ of degree exactly $d-1$. Let $\hat{w}_n$ be a vector containing all monomials from $w_{n+1,d-1}$  that include up to $n$ variables in $x$ and $\hat{w}_{n+1}$ be a vector containing all monomials from $w_{n+1,d-1}$ with exactly $n+1$ variables in $x$.} Obviously, $w_{n+1,d-1}$ is equal to $$\hat{w}\mathrel{\mathop{:}}=\begin{pmatrix} \hat{w}_n \\ \hat{w}_{n+1} \end{pmatrix}$$ up to a permutation of its entries. If we show that there exists a strictly dd matrix $\hat{Q}$ such that 
\begin{align} \label{eq:induction.permutation}
y^TH_{p_{n+1,2d}}(x)y=\hat{w}^T(x,y)\hat{Q}\hat{w}(x,y)
\end{align}
{\gh then one can easily construct a strictly dd matrix $Q_{p_{n+1,2d}}$ such that (\ref{eq:induction.result}) will hold by simply permuting the rows of $\hat{Q}$ appropriately.}

{\gh We now show the existence of such a $\hat{Q}$. To do this, we claim and prove the following:}
\begin{itemize}
\item \emph{Claim 1:} there exists a strictly dd matrix $\hat{Q}_q$ such that
\begin{align}\label{eq:permut.q}
y^TH_{q}(x)y=\begin{pmatrix} \hat{w}_n \\ \hat{w}_{n+1} \end{pmatrix}^T \begin{pmatrix} \hat{Q}_q & 0 \\ 0 & 0 \end{pmatrix} \begin{pmatrix} \hat{w}_n \\ \hat{w}_{n+1} \end{pmatrix}.
\end{align} 
\item \emph{Claim 2:} there exist a symmetric matrix $\hat{Q}_{v}$,  and $q_1,\ldots,q_m>0$ (where $m$ is the length of $\hat{w}_{n+1}$) such that
\begin{align}\label{eq:permut.gam}
y^TH_{v}(x)y=\begin{pmatrix} \hat{w}_n \\ \hat{w}_{n+1} \end{pmatrix}^T \begin{pmatrix} \hat{Q}_{v} & 0 \\ 0 & \text{diag}(q_1,\ldots,q_m)\end{pmatrix} \begin{pmatrix} \hat{w}_n \\ \hat{w}_{n+1} \end{pmatrix}.
\end{align}
\end{itemize}

{\gh Using these two claims and the fact that $p_{n+1,2d}=q+\alpha v$, we get that $$y^TH_{p_{n+1,2d}}(x)y=y^TH_q(x)y+\alpha y^T H_v(x)y=\hat{w}^T(x,y)\hat{Q}\hat{w}(x,y)$$ where $$\hat{Q}=\begin{pmatrix} \hat{Q}_q+\alpha \hat{Q}_v & 0 \\ 0 & \alpha \text{ diag}(q_1,\ldots,q_m) \end{pmatrix}.$$} As $\hat{Q}_q$ is strictly dd, we can pick $\alpha>0$ small enough such that $\hat{Q}_q+\alpha \hat{Q}_v$ is strictly dd. This entails that $\hat{Q}$ is strictly dd, and  (\ref{eq:induction.permutation}) holds.

{\gh It remains to prove the two claims to be done.}\\

{\gh \emph{Proof of Claim 1:} Claim 1 concerns the polynomial $q$, defined as the sum of polynomials $p_{n,2d}(x_{i_1},\ldots,x_{i_n})$. Note from (\ref{eq:induction.hyp}) that the Hessian of each of these polynomials has a strictly dd Gram matrix in the monomial vector $w_{n,d-1}.$ However, the statement of Claim 1 involves the monomial vector $\hat{w}_n$. So, we start by linking the two monomial vectors.}
If we denote by $$M=\mathop{\cup}_{(i_1,\ldots,i_n) \in \{1,\ldots,n+1\}^n} \{\text{monomials in } w_{n,d-1}(x_{i_1},\ldots, x_{i_n},y)\},$$ then $M$ is exactly equal to $\hat{M}=\{\text{monomials in }\hat{w}_n(x,y)\}$ {\gh as the entries of both are monomials of degree 1 in $y$ and of degree $d-1$ and in $n$ variables of $x=(x_1,\ldots,x_{n+1}).$

By definition of $q$, we have that 
\scalefont{0.69}
\begin{align*}
y^TH_qy=\sum_{(i_1,\ldots,i_n) \in \{1,\ldots,n+1\}^n} w_{n,d-1}(x_{i_1}, \ldots,x_{i_n},y)^T Q_{p_{n,2d}(x_{i_1},\ldots,x_{i_n})} w_{n,d-1}(x_{i_1}, \ldots,x_{i_n},y)
\end{align*}
\normalsize

We now claim that there exists a strictly dd matrix $\hat{Q}_q$ such that $$y^TH_qy=\hat{w}_n^T \hat{Q}_q \hat{w}_n.$$ This matrix is constructed by padding the strictly dd matrices $Q_{p_{n,2d}(x_{i_1},\ldots,x_{i_n})}$ with rows of zeros and then adding them up. The sum of two rows that verify the strict diagonal dominance condition still verifies this condition. So we only need to make sure that there is no row in $\hat{Q}_q$ that is all zero. This is indeed the case because $\hat{M} \subseteq M.$}

{\gh \emph{Proof of Claim 2:}} Let $I \mathrel{\mathop{:}}=\{i_1,\ldots,i_n~|~ i_1+\ldots+i_{n+1}=d, i_1,\ldots,i_{n+1}>0\}$ and $\hat{w}_{n+1}^i$ be the $i^{th}$ element of $\hat{w}_{n+1}$. To prove (\ref{eq:permut.gam}), we need to show that
\begin{equation}\label{eq:hessian.gam.ondiag}
\begin{aligned}
y^TH_{v}(x)y=\sum_{i_1,\ldots,i_n \in I}\sum_{k=1}^{n+1}2i_k(2i_{k}-1)x_1^{2i_1}\ldots x_k^{2i_k-2}\ldots x_{n+1}^{2i_{n+1}}y_k^2 \\
+4\sum_{i_1,\ldots,i_m \in I} \sum_{j \neq k} i_k i_j x_1^{2i_1}\ldots x_j^{2i_j-1}\ldots x_k^{2i_k-1}\ldots x_{n+1}^{2i_{n+1}}y_jy_k.
\end{aligned}
\end{equation}
can equal
\begin{align} \label{eq:expansion}
\hat{w}_n^T(x,y)\hat{Q}_v \hat{w}_n^T(x,y)+\sum_{i=1}^m q_{i}(\hat{w}_{n+1}^i)^2
\end{align}
for some symmetric matrix $\hat{Q}_{v}$ and positive scalars $q_1,\ldots, q_m$.
{\gh We first argue that all monomials contained in $y^TH_v(x)y$ appear in the expansion (\ref{eq:expansion}). This means that we do not need to use any other entry of the Gram matrix in (\ref{eq:permut.gam}). Since every monomial appearing in the first double sum of (\ref{eq:hessian.gam.ondiag}) involves only even powers of variables, it can be obtained via the diagonal entries of $Q_v$ together with the entries $q_1,\ldots,q_m.$ Moreover, since the coefficient of each monomial in this double sum is positive and since the sum runs over all possible monomials consisting of even powers in $n+1$ variables, we conclude that $q_i>0$, for $i=1,\ldots,m.$


Consider now any monomial contained in the second double sum  of (\ref{eq:hessian.gam.ondiag}). We claim that any such monomial can be obtained from off-diagonal entries in $\hat{Q}_v.$ To prove this claim, we show that it can be written as the product of two monomials $m'$ and $m''$ with $n$ or fewer variables in $x=(x_1,\ldots,x_{n+1})$. Indeed, at least two variables in the monomial must have degree less than or equal to $d-1$.} Placing one variable in $m'$ and the other variable in $m''$ and then filling up $m'$ and $m''$ with the remaining variables (in any fashion as long as the degrees at $m'$ and $m''$ equal $d-1$) yields the desired result. \qed

\end{proof}
\begin{proof}[of Theorem \ref{th:mainth}] Let $p_{n,2k} \in \mathcal{H}_{n,2k}$ be the form constructed in the proof of Lemma \ref{lemma:induction} which is in the interior of $\Sigma_D C_{n,2k}.$ Let $Q_{k}$ denote the strictly diagonally dominant matrix which was constructed to satisfy $$y^TH_{p_{n,2k}}y=w_{n,2k}^T(x,y)Q_{k}w_{n,2k}.$$ To prove Theorem \ref{th:mainth}, we take $$p \mathop{\mathrel{:}}=\sum_{k=1}^{d} p_{n,2k} \in \tilde{\mathcal{H}}_{n,2d}.$$ We have
\begin{align*}
y^TH_p(x)y&=\begin{pmatrix} w_{n,1}(x,y) \\ \vdots \\ w_{n,d-1}(x,y) \end{pmatrix}^T \begin{pmatrix} Q_1 & & \\ & \ddots &  \\ & & Q_d \end{pmatrix} \begin{pmatrix} w_{n,1}(x,y) \\ \vdots \\ w_{n,d-1}(x,y) \end{pmatrix}\\
&=\tilde{w}_{n,d-1}(x,y)^TQ\tilde{w}_{n,d-1}(x,y).
\end{align*} We observe that $Q$ is strictly dd, which shows that $p \in int(\tilde{\Sigma}_DC_{n,2d}).$ \qed
\end{proof}

\begin{remark}
If we had only been interested in showing that any polynomial in $\tilde{\mathcal{H}}_{n,2d}$ could be written as a difference of two sos-convex polynomials, this could have been easily done by showing that $p(x)=\left( \sum_i x_i^2 \right)^d \in int(\Sigma C_{n,2d})$. However, this form is not dsos-convex or sdsos-convex for all $n,d$ (e.g., for $n=3$ and $2d=8$). We have been unable to find a simpler proof for existence of sdsos-convex dcds that does not go through the proof of existence of dsos-convex dcds.
\end{remark}
\begin{remark}
If we solve problem (\ref{eq:undom.dcd}) with the convexity constraint replaced by a dsos-convexity (resp. sdsos-convexity, sos-convexity) requirement, the same arguments used in the proof of Theorem \ref{thm:undom.dcd} now imply that the optimal solution $g^*$ is not dominated by any dsos-convex (resp. sdsos-convex, sos-convex) decomposition.
\end{remark} 

\section{Numerical results}\label{sec:numerical.results}
In this section, we present a few numerical results to show how our algebraic decomposition techniques affect the convex-concave procedure. The objective function $p \in \tilde{\mathcal{H}}_{n,2d}$ in all of our experiments is generated randomly following the ensemble of \cite[Section 5.1.]{Minimize_poly_Pablo}. This means that $$p(x_1,\ldots,x_n)=\sum_{i=1}^n x_i^{2d}+g(x_1,\ldots,x_n),$$ where $g$ is a random polynomial of total degree $\leq 2d-1$ whose coefficients are random integers uniformly sampled from $[-30,30].$ An advantage of polynomials generated in this fashion is that they are bounded below and that their minimum $p^*$ is achieved over $\mathbb{R}^n.$ {\gh We have intentionally restricted ourselves to polynomials of degree equal to $4$ in our experiments as this corresponds to the smallest degree for which the problem of finding a dc decomposition of $f$ is hard, without being too computationally expensive. Experimenting with higher degrees however would be a worthwhile pursuit in future work.} The starting point of CCP was generated randomly from a zero-mean Gaussian distribution. 

One nice feature of our decomposition techniques is that all the polynomials $f_i^k, i=0,\ldots, m$ in line 4 of Algorithm \ref{alg:CCP} in the introduction are sos-convex. This allows us to solve the convex subroutine of CCP exactly via a single SDP \cite[Remark 3.4.]{Monique_Etienne_Convex}, \cite[Corollary 2.3.]{Lasserre_Jensen_inequality}:
\begin{equation} \label{eq:Lasserre.hierarchy}
\begin{aligned}
&\min \gamma \\
&\text{s.t. } f_0^k -\gamma=\sigma_0+\sum_{j=1}^m \lambda_j f_j^k\\
&\sigma_0 \text{ sos}, ~\lambda_j \geq 0, j=1,\ldots,m.
\end{aligned}
\end{equation}
The degree of $\sigma_0$ here is taken to be the maximum degree of $f_0^k,\ldots, f_m^k$. We could have also solved these subproblems using standard descent algorithms for convex optimization. However, we are not so concerned with the method used to solve this convex problem as it is the same for all experiments. All of our numerical examples were done using MATLAB, the polynomial optimization library SPOT \cite{SPOT_Megretski}, and the solver MOSEK \cite{mosek}.

\subsection{Picking a good dc decomposition for CCP} \label{subsec:numexpOneDecomp}
In this subsection, we consider the problem of minimizing a random polynomial $f_0 \in \tilde{\mathcal{H}}_{8,4}$ over a ball of radius $R$, where $R$ is a random integer in $[20,50].$ The goal is to compare the impact of the dc decomposition of the objective on the performance of CCP. To monitor this, we decompose the objective in 4 different ways and then run CCP using the resulting decompositions. These decompositions are obtained through different SDPs that are listed in Table~\ref{tab:diff.objs}.

\begin{table}[h!]
\centering
\begin{tabular}{|c|c|c|c|}
\hline
Feasibility & $\lambda_{\max}H_h(x_0)$ & $\lambda_{\max,B} H_h$ & Undominated\\
\hline
 & $\min t $ & $\min_{g,h} t$ & \\
$\min 0$ &  $\text{s.t. } f_0=g-h,$ &  $\text{s.t. } f_0=g-h,$ & $\min\frac{1}{\mathcal{A}_n}\int \mbox{Tr} H_g d\sigma$ \\
$\text{s.t. } f_0=g-h,$ & $g,h$ sos-convex & $g,h$ sos-convex &  $\text{s.t. } f_0=g-h,$ \\
$g,h$ sos-convex & $tI-H_{h}(x_0) \succeq 0 $ & $y^T(tI-H_h(x)+f_1 \tau(x))y$ sos & $g,h$ sos-convex\\
& & $y^T \tau(x) y$ \tablefootnote{Here, $\tau(x)$ is an $n \times n$ matrix where each entry is in $\tilde{\mathcal{H}}_{n,2d-4}$}  sos& \\
\hline
\end{tabular}
\caption{Different decomposition techniques using sos optimization}
\label{tab:diff.objs}
\vspace{-10pt}
\end{table}

The first SDP in Table \ref{tab:diff.objs} is simply a feasibility problem. The second SDP minimizes the largest eigenvalue of $H_h$ at the initial point $x_0$ inputed to CCP. The third minimizes the largest eigenvalue of $H_h$ over the ball $B$ of radius $R$. Indeed, let $f_1 \mathrel{\mathop{:}}=\sum_i x_i^2-R^2.$ Notice that $\tau(x) \succeq 0, \forall x$ and if $x \in B$, then $f_1(x) \leq 0$. This implies that $tI \succeq H_h(x), \forall x \in B.$ The fourth SDP searches for an undominated dcd.

Once $f_0$ has been decomposed, we start CCP. After $4$ mins of total runtime, the program is stopped and we recover the objective value of the last iteration. This procedure is repeated on 30 random instances of $f_0$ and $R$, and the average of the results is presented in Figure \ref{fig:Onedecomp}.
\begin{figure}[h!]
\centering
\includegraphics[scale=0.45]{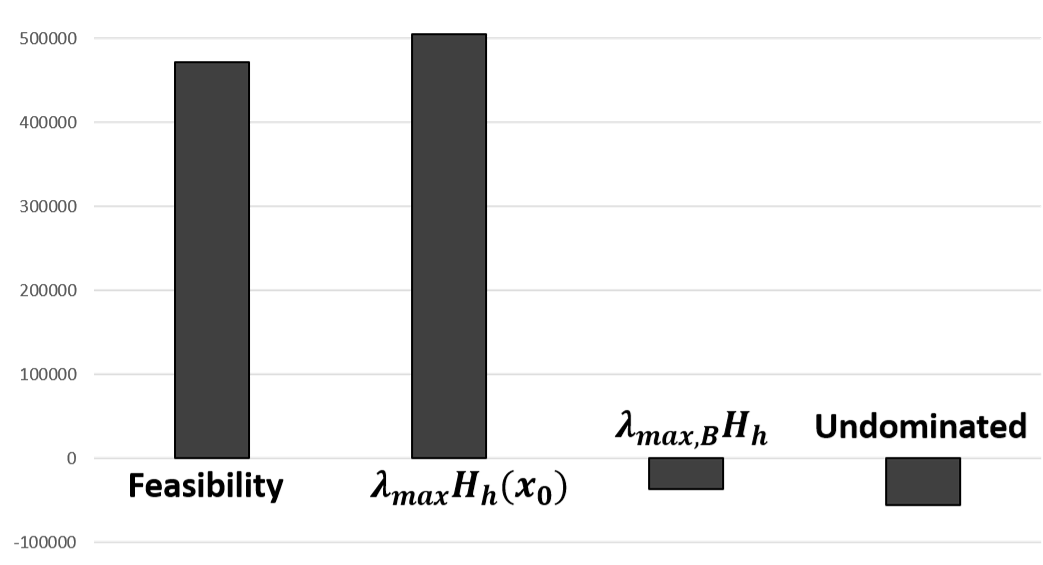}
\caption{Impact of choosing a good dcd on CCP ($n=8,2d=4$)}
\label{fig:Onedecomp}
\end{figure}
From the figure, we can see that the choice of the initial decomposition impacts the performance of CCP considerably, with the region formulation of $\lambda_{\max}$ and the undominated decomposition giving much better results than the other two. It is worth noting that all formulations have gone through roughly the same number of iterations of CCP (approx. 400). Furthermore, these results seem to confirm that it is best to pick an undominated decomposition when applying CCP.

\subsection{Scalibility of s/d/sos-convex dcds and the multiple decomposition CCP}\label{subsec:scalibility}

While solving the last optimization problem in Table \ref{tab:diff.objs} usually gives very good results, it relies on an sos-convex dc decomposition. However, this choice is only reasonable in cases where the number of variables and the degree of the polynomial that we want to decompose are low. When these become too high, obtaining an sos-convex dcd can be too time-consuming. The concepts of dsos-convexity and sdsos-convexity then become interesting alternatives to sos-convexity.
This is illustrated in Table \ref{tab:dsos}, where we have reported the time taken to solve the following decomposition problem:
\begin{equation} \label{eq:decomp.relax}
\begin{aligned}
&\min \frac{1}{\mathcal{A}_n} \int_{S^{n-1}} \mbox{Tr } H_g d\sigma\\
&\text{s.t. } f=g-h, g,h \text{ s/d/sos-convex}
\end{aligned}
\end{equation}
In this case, $f$ is a random polynomial of degree $4$ in $n$ variables. We also report the optimal value of (\ref{eq:decomp.relax}) (we know that (\ref{eq:decomp.relax}) is always guaranteed to be feasible from Theorem \ref{th:mainth}).
\begin{table}[h!] 
\begin{tabular}{|c|c|c|c|c|c|c|c|c|}
\hline
&  \multicolumn{2}{|c|}{n=6} & \multicolumn{2}{|c|}{n=10} & \multicolumn{2}{|c|}{n=14} & \multicolumn{2}{|c|}{n=18}\\
 & Time & Value & Time & Value & Time & Value & Time & Value\\
\hline dsos-convex & $<1s$ & 62090 & $<$1s & 168481 & 2.33s &136427 & 6.91s & 48457 \\
\hline sdsos-convex & $<1s$ & 53557 & 1.11 s& 132376& 3.89s &99667 & 12.16s & 32875 \\
\hline sos-convex & $<1s$ & 11602 & 44.42s &18346  &800.16s &9828 & 30hrs+ & ------\\
\hline
\end{tabular}
\caption{Time and optimal value obtained when solving (\ref{eq:decomp.relax})}
\label{tab:dsos}
\vspace{-15pt}
\end{table}
Notice that for $n=18$, it takes over 30 hours to obtain an sos-convex decomposition, whereas the run times for s/dsos-convex decompositions are still in the range of 10 seconds. This increased speed comes at a price, namely the quality of the decomposition. For example, when $n=10$, the optimal value obtained using sos-convexity is nearly 10 times lower than that of sdsos-convexity. 

Now that we have a better quantitative understanding of this tradeoff, we propose a modification to CCP that leverages the speed of s/dsos-convex dcds for large $n$.
The idea is to modify CCP in such a way that one would compute a new s/dsos-convex decomposition of the functions $f_i$ after each iteration. Instead of looking for dcds that would provide good global decompositions (such as undominated sos-convex dcds), we look for decompositions that perform well locally. From Section \ref{sec:undominated}, candidate decomposition techniques for this task can come from formulations (\ref{eq:trace.point}) and (\ref{eq:lambda.max.point}) that minimize the maximum eigenvalue of the Hessian of $h$ at a point or the trace of the Hessian of $h$ at a point. This modified version of CCP is described in detail in Algorithm \ref{alg:itCCP}. We will refer to it as \emph{multiple decomposition CCP}.

We compare the performance of CCP and multiple decomposition CCP on the problem of minimizing a polynomial $f$ of degree 4 in $n$ variables, for varying values of $n$. In Figure \ref{fig:sdsos.vs.sos}, we present the optimal value (averaged over 30 instances) obtained after 4 mins of total runtime. The ``SDSOS" columns correspond to multiple decomposition CCP (Algorithm \ref{alg:itCCP}) with sdsos-convex decompositions at each iteration. The ``SOS" columns correspond to classical CCP where the first and only decomposition is an undominated sos-convex dcd. From Figure \ref{fig:Onedecomp}, we know that this formulation performs well for small values of $n$. This is still the case here for $n=8$ and $n=10$. However, this approach performs poorly for $n=12$ as the time taken to compute the initial decomposition is too long. In contrast, multiple decomposition CCP combined with sdsos-convex decompositions does slightly worse for $n=8$ and $n=10$, but significantly better for $n=12$.

\begin{algorithm}[h]
\caption{ Multiple decomposition CCP ($\lambda_{\max}$ version)}
\label{alg:itCCP}
\begin{algorithmic}[1]
\Require $x_0,~ f_i, i=0,\ldots,m$
\State $k\leftarrow 0$
\While{stopping criterion not satisfied}
\State Decompose: $\forall i$ find $g_i^k,h_i^k$ s/d/sos-convex that min. $t$, s.t. $tI-H_{h_i^k}(x_k)$ s/dd\footnotemark and $f_i=g_i^k-h_i^k$
\State Convexify: $f_i^{k}(x)\mathrel{\mathop{:}}=g_i^k(x)-(h_i^k(x_k)+\nabla h_i^k(x_k)^T(x-x_k)),~ i=0,\ldots,m$
\State Solve convex subroutine: $\min f_0^k(x)$, s.t. $f_i^k(x) \leq 0, i=1,\ldots,m$
\State $x_{k+1}\mathrel{\mathop{:}}= \underset{f_i^{k}(x) \leq 0}{\text{argmin}} f_0^k(x)$
\State $k \leftarrow k+1$
\EndWhile
\Ensure $x_k$
\end{algorithmic}
\end{algorithm}
\footnotetext{Here dd and sdd matrices refer to notions introduced in Definition \ref{def:dd.sdd}. Note that any $t$ which makes $tI-A$ dd or sdd gives an upperbound on $\lambda_{\max}(A).$ By formulating the problem this way (instead of requiring $tI\succeq A$) we obtain an LP or SOCP instead of an SDP.}

\begin{figure}[h!]
\centering
\includegraphics[scale=0.5]{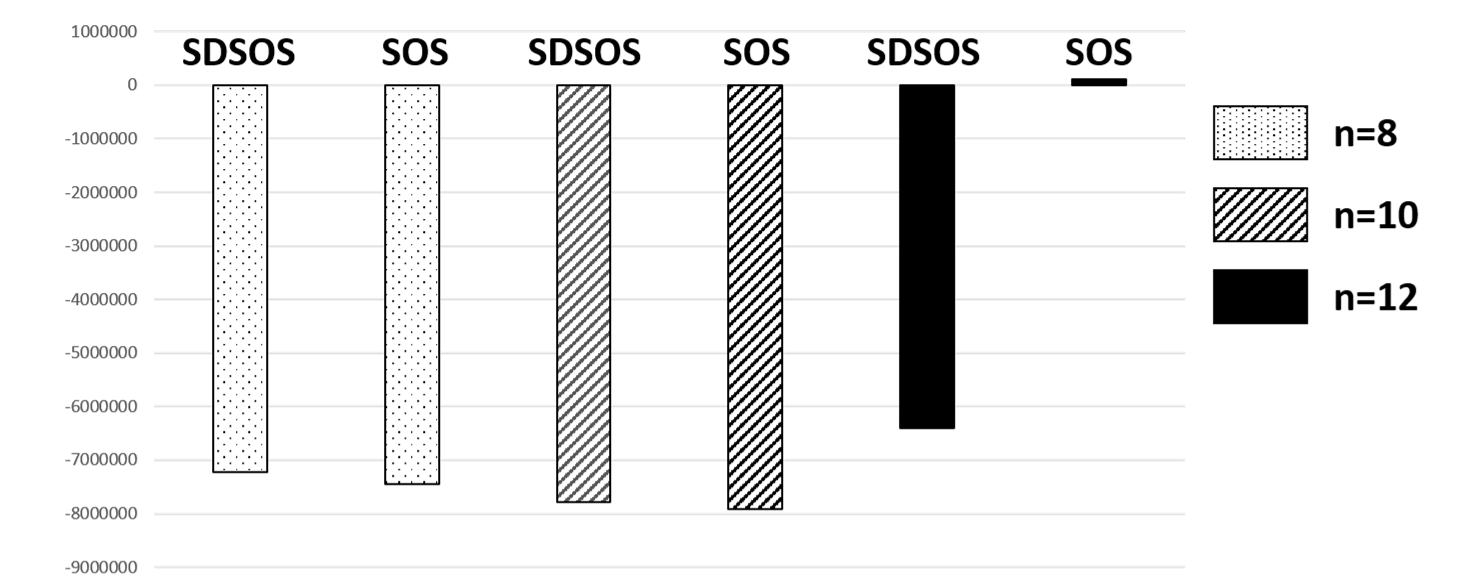}
\caption{Comparing multiple decomposition CCP using sdsos-convex decompositions against CCP with a single undominated sos-convex decomposition}
\label{fig:sdsos.vs.sos}
\vspace{-15pt}
\end{figure}

In conclusion, our overall observation is that picking a good dc decomposition noticeably affects the perfomance of CCP. While optimizing over all dc decompositions is intractable for polynomials of degree greater or equal to $4$, the algebraic notions of sos-convexity, sdsos-convexity and dsos-convexity can provide valuable relaxations. The choice among these options depends on the number of variables and the degree of the polynomial at hand. Though these classes of polynomials only constitute subsets of the set of convex polynomials, we have shown that even the smallest subset of the three contains dcds for any polynomial.

\begin{acknowledgements}
We would like to thank Pablo Parrilo for insightful discussions and Mirjam D\"{u}r for pointing out reference \cite{Bomze}.
\end{acknowledgements}

\bibliographystyle{spmpsci}      



\bibliography{pablo_amirali}


\end{document}